\documentclass[twoside,12pt]{article}
\usepackage{amsmath, amsthm, amscd, amsfonts, amssymb, graphicx, color}
\usepackage{graphicx}
\usepackage[square,numbers,sort&compress]{natbib}
\begin{document}
\setcounter{page}{1}
\setlength{\unitlength}{12mm}
\newcommand{\f}{\frac}
\newtheorem{theorem}{Theorem}[section]
\newtheorem{lemma}[theorem]{Lemma}
\newtheorem{proposition}[theorem]{Proposition}
\newtheorem{corollary}[theorem]{Corollary}
\theoremstyle{definition}
\newtheorem{definition}[theorem]{Definition}
\newtheorem{example}[theorem]{Example}
\newtheorem{solution}[theorem]{Solution}
\newtheorem{notation}[theorem]{Notation}
\theoremstyle{remark}
\newtheorem{remark}[theorem]{Remark}
\numberwithin{equation}{section}
\newcommand{\sta}{\stackrel}
\title{\bf Generalization of two-dimensional Hardy type inequality for fuzzy integrals}
\author{Bayaz Daraby$^{a}$, Mortaza Tahmourasi$^{b}$ and Asghar Rahimi$^c$}
\date{\footnotesize $^{a, b, c}$ Department of Mathematics, University of Maragheh, P. O. Box 55136-553, Maragheh, Iran\\ $^{a}$ Corresponding Author
}
\maketitle


\begin{abstract}
In this paper, a new two-dimensional Hardy type inequality is given in terms of pseudo-analysis dealing with set-valued functions. The first one is given for a pseudo-integral of set-valued function where pseudo-addition and pseudo-multiplication are constructed by a monotone continuous function $g:[0, \infty ]\to[0, \infty]$. Another is given by the semiring $([0, 1], \max, \odot)$ where pseudo-multiplication is generated by an increasing continuous function.
\end{abstract}
\begin{footnotesize}
 Pseudo integrals.
\end{footnotesize}

\section{Introduction}
$\quad$ Firstly, we express the concept of pseudo-analysis. The concept of pseudo-analysis is derived from classical analysis, which is one of the most widely used and interesting generalizations of classical analysis, which is based on the structure of semirings on the real interval
$[a, b]\subseteq[-\infty, +\infty]$
 with pseudo-addition and pseudo-multiplication operators (see \cite{pap2}).

One of the advantages of pseudo-analysis is it's wider scope, which can include nonlinear and indeterminate problems from different branches, as well as the use of mathematical tools in various fields. Based on the semiring structure in pseudo-analysis, the concepts of pseudo-measure and pseudo-integral have been developed, and accordingly many classical integral inequalities relative to pseudo-sum  have been extended.

Newly, the researchers in \cite{salem} presented a two-dimensional Hardy type inequality for fuzzy integrals. Two-dimensional Hardy type integrals have wide applications in the Fourier transform, in strong maximal functions and the double Hilbert transform (see \cite{Muc}). For seeing more applications, we refer to \cite{Hei, Saw}.  Daraby et. al. have popularized some fuzzy integral inequalities for  the Sugeno integrals and pseudo-integrals in \cite{5, dar, 6, 0, dar-sha}. Our paper generalize  Román-Flores et. al. work's \cite{salem}  for pseudo-integrals.

In the classical mathematical analysis, the Hardy type inequality is as follows (\cite{har}):
If $p>1$ and $f: [0, \infty)\to[0, \infty)$ is defined and integrable function $(f\ne 0)$ and
$F(x)=
 \int_0^x f(t)dt$, then the classical Hardy type
 integral inequality is as follows:
 $$ \left(\frac {p}{p-1}\right)^p \int_0^\infty f^p(x) dx >
  \int_0^\infty \left(\frac{F}{x}\right)^p dx,\quad\quad\quad$$
such that $ 0 < a < b < \infty $. Also, $f^p$ is integrable on $[0, \infty)$
 $$ \left(\frac{p}{p-1}\right)^p \int_a^b f^p (x) dx >\int_a^b  \left(\frac{F}{x} \right)^p dx.$$
Rom\'{a}n-Flores et. al. expressed and proved the Hardy type integral inequality for fuzzy integrals as follows  (\cite{rrom}):
 \begin{eqnarray}
\left(-  \hspace{-1.1em}\int_0^1 f^p(x) dx\right)^{\frac{1}{p+1}} \geq  -  \hspace{-1.1em}\int_0^1
 \left(\frac{F}{x}\right)^p dx
 \end{eqnarray}
 where $p\geq 1,$
  $f: [0 ,  1]\to[0 ,  \infty)$  is an integrable function and $F(x)= -  \hspace{-.8em} \int_0^x f(t)dt.$

As well as, Salem \cite{salem1} has given two-dimensional version of Hardy type inequality where $p>1$ is a  constant and $f(x, y)$ is a non-negative and
%
integrable function on $(0, \infty)^2$. Assuming
\begin{eqnarray}
\label{n7}
R(x, y)=\dfrac{1}{xy}\int_0^x\int_0^y f(s, t)dtds,
\end{eqnarray}
we have
\begin{eqnarray}
\label{n8}
\int_0^\infty\int_0^\infty R^p(x, y)dxdy\le\left(\dfrac{p}{p-1}\right)^{2p}\int_0^\infty\int_0^\infty f^p(x, y)dxdy.
\end{eqnarray}

Rom\'{a}n-Flores et. al. stated and proved two-dimensional fuzzy Hardy type inequality as follows:
\begin{theorem}\label{rtgfds}(\cite{salem})
 Let $f:[0, 1]^2\to[0, \infty]$ be an integrable functions and
\begin{eqnarray*}
\label{n7}
R(x, y)=-  \hspace{-1em} \int_{[0, x]\times[0, y]}fd(\mu\times\mu).
\end{eqnarray*}
Then
\begin{eqnarray}
\label{n8}
\left(-  \hspace{-1em} \int_{[0, 1]^2} f^p(x, y)d(\mu\times \mu) \right)^{\frac{1}{2p+1}}\ge\left(\dfrac{4}{5}\right)^{\frac{16p}{9(2p+1)}}-  \hspace{-1em} \int_{[0, 1]^2}\left(\dfrac{R(x, y)}{xy}\right)^p d(\mu\times\mu),
\end{eqnarray}
for all $p\ge 1$.
\end{theorem}

Our goal of  this paper is to prove two-dimensional Hardy type inequality for the pseudo-integrals.
 Main results are expressed with the proofs and illustrated by some examples. 

\section{Preliminaries}
~~~~This part of the paper are presented by an  important traits of pseudo-operations and pseudo-additive measures that mentioned in \cite{mes, rrom, 21}.

Let $[a, b]$ be a closed or semiclosed subinterval of $[-\infty, \infty]$. The full order on
$[a, b]$ will be denoted by $\preceq$.

Let $[a, b]_+ = \left\{x | x \in [a, b], {\bf 0} \preceq x\right\}$.



In \cite{21} the operations $\oplus$ and $\odot$ are defined.
Those operations are named pseudo-addition and pseudo-multiplication respectively.
The operations $\oplus$ is a commutative, non decreasing function (with respect to $\preceq$), associative and with a zero (netural) element indicated by \textbf{0}. The operation $\odot$ is a commutative, positively non decreasing function, associative and for each $x\in[a, b]$, $\textbf{1}\odot x=x$.
Also, we assume ${\bf 0} \odot x = {\bf 0}$  that $\odot$ is a
  distributive pseudo-multiplication with respect to $\oplus$.

  Case I: The pseudo-addition is idempotent operation and the pseudo-multiplication is not.\\
Case II: The pseudo-addition and pseudo-multiplication are defined by a monotone and continuous
  function $g: [a, b] \to [0, \infty]$, i.e., pseudo-operations are given with $x \oplus y = g^{-1}\big(g\left(x\right) +g\left(y\right)\big)$ and $x \odot y = g^{-1}\big(g(x)g(y)\big)$.

Case III: Both operations are idempotent. 
  For example $x\oplus y=\sup(x, y)$, $x\odot y=\inf(x, y)$ on the interval $[a, b]$.
  
In the most of the paper, we consider the semiring $([0, 1], \oplus, \odot )$ for two significant cases. The first case is when pseudo-operations are produced by a monotone and continuous function such as $g:[a, b]\to[0, \infty)$.
Therefore, the pseudo-integral for a function $f:[c, d]\to[a, b]$ scale down the $g-$integral
\begin{eqnarray}\label{31}
\int_{[c, d]}^\oplus f(x)dx = g^{-1}\left(\int_c^d g(f(x))dx\right).
\end{eqnarray}
The second class is when $x \oplus y= \sup(x, y)$ and $x \odot y= g^{-1} (g(x)g(y)),$ the  pseudo-integral for a function $f : \mathbb{R}  \to [a,b]$ be given as follows:
$$\int_\mathbb{R}^{\sup} f \odot dm= \sup_{x\in\mathbb{R}} \left( f(x) \odot \psi(x) \right),$$
in which the function $\psi$ defines sup-measure $m.$
\begin{theorem}\label{t1.5}
(\cite{pap2})  For two measurable functions $f, f_1, f_2$
  and $\lambda \in \mathbb{R},$ we have

 (i) $\int_{[c, d]}^\oplus (f_1 \oplus f_2)dx = \int_{[c, d]}^\oplus f_1 dx \oplus \int_{[c, d]}^\oplus f_2dx,$

(ii) $\int_{[c, d]}^\oplus (\lambda \otimes f)dx = \lambda  \otimes \int_{[c, d]}^\oplus  f dx,$

(iii) $f_1 \leq f_2 \Longrightarrow \int_{[c, d]}^\oplus f_1 dx \leq \int_{[c, d]}^\oplus  f_2dx.$

 (iv) If $[a, b]\subseteq[c, d]$, then $\int_{[a, b]}^\oplus f(x)dx\le \int_{[c, d]}^\oplus f(x)dx$.
\end{theorem}

In \cite{dar-har}, Daraby et. al. proved the following Lemmas.
\begin{lemma}\label{lp}
If $ f: [0, 1] \to [0, 1] $ is a
  $\mu$-measurable function and $ g:  [0, 1] \to [0, 1] $ is a
  continuous  function, then
  \begin{eqnarray}
  \int_{[0, 1]}^\oplus f^s
  d\mu \geq \left(\int_{[0, 1]}^\oplus f d\mu \right)^s
  \end{eqnarray}
   holds for all $ s\geq 1$.
  \end{lemma}

   \begin{lemma}\label{l23}
 Let $ f$ be defined from $[0, 1]$ to $[0, 1] $ and be a
  continuous  function. If $m$ is the same as in Theorem 1, \cite{mes}, and $g:[0, 1] \to [0, 1] $ is a continuous generator function, thereupon
\begin{eqnarray}
\label{land}
\left(\int_{[0, 1]}^{\sup} f dm \right)^s  \le \int_{[0, 1]}^{\sup} f^s  dm,
\end{eqnarray}
  holds, where $ s\in[1, \infty)$.
 \end{lemma}

 In the following, using the Theorem \ref{t1.5}. (iii), we have the following corollary.  Note that due to the similarity of the proof of Theorem \ref{t1.5}, we omitted the proof.

\begin{corollary}\label{tt1.5}
For two measurable functions $f, h$ and $a, b, c, d \in\mathbb{R}$, 
we have:

 if $f\le h$, then 
$$\int_{[a, b]}^{\oplus}\int_{[c, d]}^{\oplus} f(x)dx\le \int_{[a, b]}^{\oplus}\int_{[c, d]}^{\oplus} h(x)dx.$$

\end{corollary}

\section{Main Results}

Two-dimensional Hardy type inequality for fuzzy integrals has proved  in \cite{salem} by H. Rom\'{a}n-Flores et. al.
In this section, we are going to state and prove two-dimensional Hardy type inequality for pseudo-integrals.
In this paper, we consider the semiring $([0, 1], \oplus, \odot)$.

\begin{theorem}\label{t1-1}
Let $f$ be defined from $[0, 1]^2$ to $[0, 1]$ and be a non-negative, increasing and integrable function and $g: [0, 1]\to[0, 1]$  be a continuous and monotone function. Then the inequality
\begin{eqnarray}
\label{0}
\int_{[0, 1]}^\oplus\int_{[0, 1]}^\oplus R^p(x, y)dxdy\le\left(\dfrac{p}{p-1}\right)^{2p}\int_{[0, 1]}^\oplus\int_{[0, 1]}^\oplus f^p(x, y)dxdy.
\end{eqnarray}
holds, where $p>1$ and
$$R(x, y)=\dfrac{1}{xy}\int_{[0, x]}^\oplus\int_{[0, y]}^\oplus f(s, t)dtds.$$
\end{theorem}\begin{proof}
Based on definition $R(x, y)$ and as we know,
$$f(x, y)\ge\sup f(s, t)\qquad\qquad \forall~ 0\le s\le x, 0\le t\le y.$$
\begin{eqnarray}\label{b*}
R(x, y) &= &
\dfrac{1}{xy}\int_{[0, x]}^\oplus \int_{[0, y]}^\oplus f(s, t)dtds\nonumber\\&\le &
\dfrac{1}{xy}\int_{[0, x]}^\oplus \int_{[0, y]}^\oplus f(x, y)dtds\nonumber\\&\le &
\dfrac{1}{xy}f(x, y)\int_{[0, x]}^\oplus \int_{[0, y]}^\oplus dtds.
\end{eqnarray}
We know the pseudo-integral inherit Lebesque integral properties, Therefore
\begin{eqnarray}
\label{bnmc}
\int_{[0, x]}^\oplus \int_{[0, y]}^\oplus dtds=xy.
\end{eqnarray}
So, from the Relations \eqref{b*} and \eqref{bnmc} we have
$$R(x, y)\le\dfrac{1}{xy}f(x, y){\int_{[0, x]}^\oplus \int_{[0, y]}^\oplus dtds}= f(x, y).$$
Hence, we can see easily
$$R^p(x, y)\le f^p(x, y),\qquad \forall p>1.$$
Now, from Proposition \ref{tt1.5},  we obtain
\begin{eqnarray}\label{4}
\int_{[0, 1]}^\oplus \int_{[0, 1]}^\oplus R^p(x, y)dxdy &\le &\int_{[0, 1]}^\oplus \int_{[0, 1]}^\oplus f^p(x, y)dxdy,
\end{eqnarray}
 and since
\begin{eqnarray}
\label{lop}
\left(\dfrac{p}{p-1}\right)^{2p}\ge 1,
\end{eqnarray}
It follows from Relations \eqref{lop} and \eqref{4},
$$\int_{[0, 1]}^\oplus\int_{[0, 1]}^\oplus R^p(x, y)dxdy\le\left(\dfrac{p}{p-1}\right)^{2p}\int_{[0, 1]}^\oplus\int_{[0, 1]}^\oplus f^p(x, y)dxdy.$$
The proof is  completed now.
\end{proof}



Now, through some examples, we show illustration of the Theorem \ref{t1-1}.

\begin{example}\label{ex1}
Let $f: [0, 1]^2\to [0, 1]$ be defined as $f(x, y)=x^2y^2$ and $g: [0, 1]\to [0, 1]$ be defined as $g(x)=\sqrt{x}$ and $p=2$. With a simple calculation, we have
\begin{eqnarray*}
R(x, y) &=&
\dfrac{1}{xy}\int_{[0, x]}^{\oplus}\int_{[0, y]}^{\oplus} f(s, t)dtds \\&=&
\dfrac{1}{xy}g^{-1}\int_0^x\int_0^y g(f(s, t))dtds\\&=&
\dfrac{1}{xy}g^{-1} \int_0^x\int_0^y (st)dtds\\&=&
\dfrac{1}{xy}\left(g^{-1}\left(\dfrac{x^2y^2}{4}\right)\right)\\&=&
\dfrac{x^3y^3}{16},
\end{eqnarray*}
and therefore
\begin{eqnarray*}
\int_{[0, 1]}^{\oplus}\int_{[0, 1]}^{\oplus} R^p(x, y)dxdy&=&
g^{-1}\int_0^1\int_0^1 g(R^p(x, y))dxdy\\&=&
g^{-1}\int_0^1\int_0^1
\sqrt{ \dfrac{x^3y^3}{16}}dxdy \\&=&g^{-1}\left(\dfrac{1}{25}\right)\\&=&\left(\dfrac{1}{25}\right)^2.
\end{eqnarray*}
Also, we compute
$$\left(\dfrac{p}{p-1}\right)^{2p}=\left(\dfrac{2}{1}\right)^{2\times 2 }=2^4,$$
and
\begin{eqnarray*}
\int_{[0, 1]}^{\oplus}\int_{[0, 1]}^{\oplus} f^p (x, y)dxdy &=&
g^{-1}\int_0^1\int_0^1 g\left({x^2y^2}\right)^2dxdy\\&=&
g^{-1}\int_0^1\int_0^1\left({x^2y^2}\right)dxdy\\&=& g^{-1}\left(\dfrac{1}{9}\right)=\dfrac{2}{9}.
\end{eqnarray*}
Thereby, we have
\begin{eqnarray*}
\left(\dfrac{1}{25}\right)^2\le 2^4\times \dfrac{2}{9}=\dfrac{2^5}{9}.
\end{eqnarray*}
So the inequality \eqref{0} is satisfying.
\end{example}

\begin{example}
Let $f: [0, 1]^2\to [0, 1]$ be defined as $f(x, y)=\dfrac{x+y}{2}$, $g: [0, 1]\to [0, 1]$ be defined as  $g(x)=\dfrac{x}{2}$ and $p=2$. From \eqref{0}, we obtain the values. Firstly, we compute
$R(x, y)$. We have
\begin{eqnarray*}
R(x, y) &=&
\dfrac{1}{xy}\int_{[0, x]}^{\oplus}\int_{[0, y]}^{\oplus} f(s, t)dtds \\&=&
\dfrac{1}{xy}g^{-1}\int_0^x\int_0^y g(f(s, t))dtds\\&=&
\dfrac{1}{xy}g^{-1} \int_0^x\int_0^y \dfrac{1}{4}(s+t)dtds\\&=&
\dfrac{1}{xy}\left(g^{-1}\left(\dfrac{xy^2+x^2y}{8}\right)\right)=\dfrac{x+y}{4},
\end{eqnarray*}
and therefore we get
\begin{eqnarray*}
\int_{[0, 1]}^{\oplus}\int_{[0, 1]}^{\oplus} R^p(x, y)dxdy&=&
g^{-1}\int_0^1\int_0^1 g(R^p(x, y))dxdy\\&=&
g^{-1}\int_0^1\int_0^1 \dfrac{1}{2}\left(\dfrac{x+y}{4}\right)^2dxdy \\&=&
g^{-1}\left(\dfrac{7}{192}\right)=\dfrac{14}{192}.
\end{eqnarray*}
Also, we have
$$\left(\dfrac{p}{p-1}\right)^{2p}=2^4.$$
In the following, we compute the first part of the equation:
\begin{eqnarray*}
\int_{[0, 1]}^{\oplus}\int_{[0, 1]}^{\oplus} f^p (x, y)dxdy &=&
g^{-1}\int_0^1\int_0^1 g\left(\dfrac{x+y}{2}\right)^2dxdy\\&=&
g^{-1}\int_0^1\int_0^1\dfrac{1}{2}\left(\dfrac{x+y}{2}\right)^2dxdy\\&=& 2\left(\dfrac{7}{48}\right)=\dfrac{7}{24}.
\end{eqnarray*}
Thereby, we have
$$\dfrac{14}{192}\le 2^4\times\dfrac{7}{24}=\dfrac{14}{3}.$$
So, from the above, the inequality \eqref{0} is shown.
\end{example}




\begin{remark}\label{re3.5}
Note that $p>1$ is a necessary condition in Theorem \ref{t1-1}.

(a) If $0<p<1$. 

According to the Example \ref{ex1} assumptions and $p=\dfrac{1}{6}$, we have
\begin{eqnarray*}
&& \left(\dfrac{p}{p-1}\right)^{2p}= \left(\dfrac{\frac{1}{6} }{\frac{-5}{6} }\right)^{\frac{1}{3}}=\left(-\dfrac{1}{5}\right)^{\frac{1}{3}}=-0.5848.
\end{eqnarray*}
Calculating as following, we have
\begin{eqnarray*}
\int_{[0, 1]}^\oplus \int_{[0, 1]}^\oplus R^p(x, y)dxdy&=& g^{-1}\int_0^1 g\left(g^{-1}\int_0^1 g\left(\dfrac{x^3y^3}{16}\right)^{\frac{1}{6}}dx\right)dy\\&=&
g^{-1} \int_0^1 \int_0^1
\sqrt{\left(\dfrac{x^3y^3}{16}\right)^{\frac{1}{6}}}
 dxdy\\&=&
g^{-1}\left(0.507968\right)
=1.015936,
\end{eqnarray*}
and
\begin{eqnarray*}
\int_{[0, 1]}^\oplus\int_{[0, 1]}^\oplus f^p(x, y)dxdy &=&
\int_{[0, 1]}^\oplus\int_{[0, 1]}^\oplus \left(x^2y^2\right)^{\frac{1}{6}}dxdy\\&=&
g^{-1}(0.734694)=1.469388.
\end{eqnarray*}
Therefore, those shows that the Inequality \eqref{0} does not hold as is written as following:
$$1.015936 \nleq (-0.5848)(1.469388).$$
(b) If $p<0$, again by coming back to the Example \ref{ex1} and letting $p=-2$, then we have
\begin{eqnarray*}
\int_{[0, 1]}^\oplus \int_{[0, 1]}^\oplus R^p(x, y)dxdy&=& g^{-1}\int_0^1 g\left(g^{-1}\int_0^1 g\left(x^2y^2\right)^{-2}dx\right)dy,
\end{eqnarray*}

that the integral does not converge.
\\(c) If $p=0$, from the Example \ref{ex1}, we must have
$$\int_{[0, 1]}^\oplus \int_{[0, 1]}^\oplus f(x, y)dxdy\ge 1,$$
but
\begin{eqnarray*}
\int_{[0, 1]}^\oplus \int_{[0, 1]}^\oplus f(x, y)dxdy&=&
\int_{[0, 1]}^\oplus \int_{[0, 1]}^\oplus x^2y^2 dxdy \\&=&
\dfrac{1}{4}= 0.25.
\end{eqnarray*}
So we conclude it must be $p>1$.
\end{remark}

In the following, we extend the Hardy type inequality using by semiring $([0, 1], \max, \odot)$, where $\odot$ is generated.

\begin{theorem}
Let $f$ be defined from $[0, 1]^2$ to $[0, 1]$ and be a $\mu$-measurable function, $g: [0, 1] \to[0, 1] $ be a continuous and monotone function. Let $m$ be the same as in Theorem 1, \cite{mes}. 
Then the inequality
\begin{eqnarray}
\label{n1}
\int_{[0, 1]}^{\sup}\int_{[0, 1]}^{\sup} R^p(x, y)dxdy\le\left(\dfrac{p}{p-1}\right)^{2p}\int_{[0, 1]}^{\sup}\int_{[0, 1]}^{\sup} f^p(x, y)dxdy,
\end{eqnarray}
holds, where $p>1$ and
$$R(x, y)=\dfrac{1}{xy}\int_{[0, x]}^{\sup}\int_{[0, y]}^{\sup} f(s, t)dtds.$$
\end{theorem}
\begin{proof}
As stated by the explanations provided in the proof of Theorem \ref{t1-1}, we prove this theorem. Based on definition $R(x, y)$ and because of $f(x, y)\ge f(s, t)$ for all $s\in[0, x]$ and $[0, y]$, we obtain that
\begin{eqnarray}\label{0b*}
R(x, y) &= &
\dfrac{1}{xy}\int_{[0, x]}^{\sup} \int_{[0, y]}^{\sup} f(s, t)dtds\nonumber\\&\le &
\dfrac{1}{xy}\int_{[0, x]}^{\sup} \int_{[0, y]}^{\sup} f(x, y)dtds\nonumber\\&\le &
\dfrac{1}{xy}f(x, y)\int_{[0, x]}^{\sup} \int_{[0, y]}^{\sup} dtds
\end{eqnarray}
Now, from Lebesque integral properties, we get
\begin{eqnarray}
\label{0bnmc}
\int_{[0, x]}^{\sup} \int_{[0, y]}^{\sup} dtds=xy.
\end{eqnarray}
So, from Relations \eqref{0b*} and \eqref{0bnmc} we have
$$R(x, y)\le\dfrac{1}{xy}f(x, y){\int_{[0, x]}^{\sup} \int_{[0, y]}^{\sup} dtds}= f(x, y).$$
Therefore, by the above relation we have 
$$\int_{[0, 1]}^{\sup}\int_{[0, 1]}^{\sup} R^p(x, y)dxdy\le\int_{[0, 1]}^{\sup}\int_{[0, 1]}^{\sup} f^p(x, y)dxdy,$$
and finally
$$\int_{[0, 1]}^{\sup}\int_{[0, 1]}^{\sup} R^p(x, y)dxdy\le\left(\dfrac{p}{p-1}\right)^{2p}\int_{[0, 1]}^{\sup}\int_{[0, 1]}^{\sup} f^p(x, y)dxdy.$$
And the proof is complete.
\end{proof}

\begin{example}
Let $f: [0, 1]^2\to [0, 1]$ be a measurable function, $g^\lambda(x)=e^{\lambda x}$ and $\psi(x)$ be the same as in Theorem 1, \cite{mes}. Then
\begin{eqnarray*}
&& x\oplus y=\lim_{\lambda\to\infty} \left(\dfrac{1}{\lambda} \ln\left(e^{\lambda x}+e^{\lambda y}\right)\right)=\max(x, y),\\&&
x\odot y=\lim_{\lambda\to\infty} \dfrac{1}{\lambda} \ln\left(e^{\lambda x}e^{\lambda y}\right)=x+y.
\end{eqnarray*}
Therefore, \eqref{n1} reduces on the following:
\begin{eqnarray*}
&& {\sup}_{x\in[0, 1]}\bigg( \Big({\sup}_{x\in[0, 1]}\big(R^p(x, y)+\psi(x)\big)\Big)+\psi(x)\bigg)\\&& \le \left(\dfrac{p}{p-1}\right)^{2p} {\sup}_{x\in[0, 1]}\Big({\sup}_{x\in[0, 1]}\big(f^p(x, y)+\psi(x)\big)+\psi(x)\Big).
\end{eqnarray*}
\end{example}

\begin{example}
Let $f: [0, 1]^2\to [0, 1]$ be a measurable function and $g^\lambda(x)=x^{-\lambda}$. We have
$$x\oplus y=\left(x^{-\lambda}+y^{-\lambda}\right)^{-1/\lambda},\qquad\qquad x\odot y=xy.$$
Therefore, \eqref{n1} reduces on the following:
\begin{eqnarray*}
&&{\sup}_{x\in[0, 1]} \bigg(\Big({\sup}_{x\in[0, 1]}\big(R^p(x, y)\psi(x)\big)\Big)\psi(x)\bigg)\\&& \le \left(\dfrac{p}{p-1}\right)^{2p} {\sup}_{x\in[0, 1]}\Big({\sup}_{x\in[0, 1]}\big(f^p(x, y)\psi(x)\big)\psi(x)\Big).
\end{eqnarray*}
\end{example}
Note that in  the third important case where $\oplus = \max$ and $\odot= \min$, it  has been studied in \cite{salem} and the pseudo-integral in such
a case yields the Sugeno integral (see the Theorem \ref{rtgfds}).

\section{Conclusion}\label{sec4}
The classical Hardy type integral inequality is one of the most important inequality
and it is deeply connected with the study of singular integral theory.
In this paper, we broght the classic and the fuzzy Hardy type inequality in two-dimensional and generalized this inequality for the pseudo-integrals. This integral inequality has wide applications in the Fourier transform, the double Hilbert transform and strong maximal functions.  In the sequel, several illustrated examples are given. Also, in Remark \ref{re3.5}, we showed that $p>1$ is a necessary condition in the Theorem \ref{t1-1}.

\date{\scriptsize $^{a}$
E-mail: bdaraby@maragheh.ac.ir,}
\date{\scriptsize $^{b}$
E-mail: mortazatahmoras@gmail.com}
\date{\scriptsize $^{c}$
E-mail: rahimi@maragheh.ac.ir
}

\end{document}